\documentclass[12pt,leqno,a4paper]{amsart}
\usepackage[latin1]{inputenc}
\usepackage{amssymb,enumerate}
\usepackage[margin=1.4in]{geometry}
\usepackage[colorinlistoftodos,bordercolor=orange,backgroundcolor=orange!20,linecolor=orange,textsize=tiny]{todonotes}
\usepackage{xspace}
\usepackage{csquotes}
\usepackage{url}
\usepackage{booktabs}

\newcommand{\CC}{{\mathbb{C}}}
\newcommand{\KK}{{\mathbb{K}}}
\newcommand{\QQ}{{\mathbb{Q}}}
\newcommand{\RR}{{\mathbb{R}}}
\newcommand{\ZZ}{{\mathbb{Z}}}

\newcommand{\fA}{{\mathfrak{A}}} 
\newcommand{\fS}{{\mathfrak{S}}} 

\newcommand{\Irr}{{\operatorname{Irr}}}
\newcommand{\PSL}{{\operatorname{L}}}
\newcommand{\PSU}{{\operatorname{U}}}
\newcommand{\PSp}{{\operatorname{S}}}

\let\vhi=\varphi

\newcommand{\GAP}{{\sf{GAP}}\xspace}
\newcommand{\polymake}{{\sf{polymake}}\xspace}
\newcommand{\polyDB}{{\sf{polyDB}}\xspace}

\newcommand{\OSCAR}{{\sf{OSCAR}}\xspace}

\newtheorem{thm}{Theorem}
\newtheorem{prop}[thm]{Proposition}
\newtheorem{lem}[thm]{Lemma}

\theoremstyle{definition}
\newtheorem{rem}[thm]{Remark}

\newtheorem{exmp}[thm]{Example}
\newtheorem{qst}[thm]{Question}

\begin{document}

\title{Zeros of $S$-characters}

\author{Thomas Breuer}
\address{Lehrstuhl f\"ur Algebra und Zahlentheorie, RWTH Aachen University,
  Pontdriesch 14/16, 52062 Aachen, Germany.}
\email{sam@math.rwth-aachen.de}

\author{Michael Joswig}
\address{Chair of Discrete Mathematics/Geometry, Technische Universit\"at
  Berlin, Germany.}
\email{joswig@math.tu-berlin.de}

\author{Gunter Malle}
\address{FB Mathematik, RPTU, Postfach 3049, 67653 Kaisers\-lautern, Germany.}
\email{malle@mathematik.uni-kl.de}

\begin{abstract}
The concept of $S$-characters of finite groups was introduced by Zhmud' as a
generalisation of transitive permutation characters. Any non-trivial
$S$-character takes a zero value on some group element. By a deep result
depending on the classification of finite simple groups a non-trivial transitive
permutation character even vanishes on some element of prime power order.
We present examples that this does not generalise to $S$-characters, thereby
answering a question posed by J-P.\ Serre.
\end{abstract}

\keywords{$S$-characters, zeros of characters, permutation characters, lattice points in polytopes}

\subjclass[2010]{%
20C15, 
20D08, 
20B05, 
11P21, 
52B20
}

\date{\today}

\maketitle

\pagestyle{myheadings}

\section{Introduction}
\label{sec:intro}

The concept of $S$-characters of finite groups was introduced by Zhmud' \cite{Zh}
as a far reaching generalisation of, for example, transitive permutation
characters. More precisely let $G$ be a finite group. Then an
\emph{$S$-character of $G$} is a virtual character $\psi$ of $G$ (i.e., an
integral linear combination of irreducible complex characters of $G$) such that
\begin{enumerate}
\item $\psi$ contains the trivial character $1_G$ exactly once, and
\item all values $\psi(g)$, $g\in G$, are non-negative real numbers.
\end{enumerate}
Examples of $S$-characters are, as mentioned before, the characters of transitive
permutation representations of $G$ which obviously satisfy the above conditions,
but also products $\psi=\chi\overline{\chi}$ for $\chi\in\Irr(G)$, where
$\overline{\chi}$ denotes the complex conjugate character of $\chi$.

Any $S$-character $\psi\ne 1_G$ vanishes on some group element \cite[Thm~1]{Zh};
this can be seen by an easy adaptation of Burnside's proof of the corresponding
statement for non-linear irreducible characters.

It was shown by Fein, Kantor and Schacher in a much cited paper \cite{FKS},
using the classification of finite simple groups, that non-trivial transitive
permutation characters even vanish on some element of prime power order. (See
also Giudici \cite{Giu} for a more complete proof of a stronger statement).
Building upon \cite{FKS} it was shown in \cite{MNO} that any non-linear
irreducible character $\chi$ of a finite group (and hence any $S$-character of
the form $\chi\overline{\chi}$) vanishes on some element of prime power order.
This naturally leads to the question, posed by J-P.~Serre \cite{Se}, whether
the conclusion continues to hold for $S$-characters.

We show that this is not the case, thereby also indicating that the conclusions
of \cite{FKS} and \cite{MNO} might be rather subtle.

\begin{thm}   \label{thm}
  The alternating group $\fA_8$ possesses a non-trivial ordinary $S$-character
  which does not vanish on any element of prime power order.
  This is the smallest example, in terms of number of conjugacy classes.
  In particular, no other group with at most 14 conjugacy classes shares that
  property.
\end{thm}

Here, an $S$-character is called \emph{ordinary} if it is an actual character.
A more comprehensive list of further such examples is given in
Section~\ref{sec:ex}. All of these are close to simple groups.
In fact, we are not aware of any solvable group as in the theorem.
\medskip

In Section~\ref{sec:poly} we explain how our problem can be reduced to the
question of finding lattice points in polyhedra. In Section~\ref{sec:solv} we
prove that every ordinary $S$-character of a solvable group vanishes on some
element of prime power order. Finally, in Section~\ref{sec:ex}, we present our
computational results with our new computer algebra system
\OSCAR~\cite{OSCAR-book,OSCAR}.
These computations include a search among groups listed in the ATLAS of Finite
Groups~\cite{Atl}. The finite groups with up to 14 conjugacy classes have been
classified by Vera L\'opez, Vera L\'opez, and
Sangroniz~\cite{VLVL85,VLVL86,VS07}.

\section{Characters as lattice points in polytopes}   \label{sec:poly}
In this section we translate the question whether a given finite groups
admits interesting $S$-characters into the language of polyhedral geometry.

\subsection{$S$-character simplices}
Let $G$ be a finite group. The complex character table of $G$ is a square matrix
$X\in\CC^{n\times n}$ whose entries lie in some algebraic (in fact, abelian)
number field. Its rows correspond to the $n$ complex irreducible characters
of~$G$, and the columns contain their values on the $n$ conjugacy classes of
elements. Complex conjugation operates both on the rows and the columns of~$X$.
Summing the row orbits yields a real valued matrix $U\in\RR^{m\times n}$, where
$m$ is the number of orbits. By omitting duplicate columns we arrive at a matrix
$V\in\RR^{m\times m}$, where the rows correspond to the $m$ irreducible
$\RR$-characters of~$G$, and the columns correspond to the orbits of conjugacy
classes, of which there are also $m$ by Brauer's permutation lemma. The entries
of $V$ lie in an abelian number field~$\KK$ which is real and thus ordered.

We consider the convex polyhedron
\begin{equation}\label{eq:polyhedron}
  S(G)\ := \ \bigl\{ x\in \KK^m \mid x V\geq0 \text{ and } x_1=1 \bigr\} \enspace,
\end{equation}
whose lattice points bijectively correspond to the $S$-characters of $G$.
Here we assume that the trivial character corresponds to the first row of $X$
and hence of $V$. In this way, the condition $x_1=1$ accounts for property~(1)
in the definition of an $S$-character.

\begin{prop}   \label{prop:simplex}
  The polyhedron $S(G)$ is a simplex of dimension $m-1$.
  If $G\neq 1$ is non-trivial, then the homogenised origin $(1,0,\ldots,0)$ is an
  interior lattice point.
\end{prop}

\begin{proof}
  It follows from \cite[Lemma~2]{Zh} that the polyhedron $S(G)$ is bounded.
  Moreover, the matrix $V$ is regular, which is why $\dim S(G)=m-1$.
  Since $S(G)$ is described in terms of $m$ inequalities, it follows that $S(G)$
  is a simplex.

  We had assumed that the trivial character occurs first.
  Since the trivial character forms its own orbit under complex conjugation,
  the first row of the matrix $V$ is the all-ones vector.
  Let $v$ be one of the $m$ columns of $V$; this is a homogeneous inward pointing
  facet normal vector of the simplex $S(G)$.
  Computing the scalar product $\langle v, e_1\rangle=1$ shows that
  $e_1=(1,0,\ldots,0)$ strictly satisfies the corresponding inequality.
  This observation holds for all facets, and so $e_1$ is an interior (lattice)
  point.
\end{proof}

\begin{rem}
  Observe that $V$ is rational-valued (and hence integer-valued) if and only if
  all entries of $X$ lie in imaginary quadratic number fields.
\end{rem}

We call $S(G)$ the \emph{$S$-character simplex} of the finite group $G$.
The interior lattice point $e_1=(1,0,\ldots,0)$ corresponds to the trivial character.
Interestingly, the inequality description \eqref{eq:polyhedron} agrees with the
homogeneous encoding of convex polyhedra which is standard to most software
systems in polyhedral geometry, including \polymake \cite{polymake-lattice}.
In the following we call the $(m{-}1)\times m$-matrix obtained by skipping the
first row (with the trivial character) in $V$ the \emph{truncated table of real
irreducible characters} of $G$, and we denote it as $V'$.
When we identify the affine hyperplane $x_1=1$ in $\RR^m$ with $\RR^{m-1}$ by
dropping the first coordinate $x_1=1$, we obtain
\begin{equation}  \label{eq:ineq}
  S(G) \ = \ \bigl\{ x \in \RR^{m-1} \mid x (-V') \leq 1 \} \enspace .
\end{equation}

\begin{exmp} \label{exmp:S4}
  The character table
  \[
    X \ = \ 
    \begin{pmatrix}
      1 & 1 & 1 & 1 & 1 \\
      1 & 1 & 1 & -1 & -1 \\
      2 & 2 & -1 & 0 & 0 \\
      3 & -1 & 0 & 1 & -1 \\
      3 & -1 & 0 & -1 & 1
    \end{pmatrix}
  \]
  of the symmetric group $\fS_4$ is rational-valued.
  So we have $m=n=5$ and $V=X$.
  The $S$-character simplex $S(\fS_4)$ is 4-dimensional with an inequality
  representation like \eqref{eq:ineq}.
  A direct computation shows that the vertices of $S(\fS_4)$ are precisely the columns of the truncated table $V'$ of real irreducible characters.
\end{exmp}

The following result will show that the situation of Example~\ref{exmp:S4} is
typical. Yet we first recall some basic concepts from polyhedral geometry.
A \emph{lattice polytope} has vertices whose coordinates are integers.
If a lattice polytope $P\subset\RR^d$ satisfies $P=\{ x\in\RR^d\mid A x\leq 1 \}$
for some integral matrix $A$, then it is called \emph{reflexive}.
Equivalently, $P$ is reflexive if and only if its polar
$P^\vee := \{ y\in\RR^d \mid \langle x,y\rangle \leq 1 \text{ for all } x\in P\}$
is a lattice polytope, too; here $\langle \cdot,\cdot\rangle$ denotes the
Euclidean scalar product. In that case the origin is the only interior lattice
point of both, $P$ and $P^\vee$; this is easy to see, but is also mentioned,
e.g., in \cite[Thm~4.1.9(iii)]{Bat}. Batyrev showed that reflexive polytopes
correspond to Fano toric varieties \cite{Bat}, which occur in string theory.
The simplex $S(\fS_4)$ from Example~\ref{exmp:S4} occurs as
\enquote{v05-000000963} in the \polyDB \cite{polydb} version\footnote{\url{https://polydb.org/#collection=Polytopes.Lattice.Reflexive}} of the classification of all four-dimensional reflexive polytopes by Kreuzer and Skarke \cite{KS}.

\begin{prop}   \label{prop:rational}
  Suppose that $G\neq 1$ is non-trivial and all characters of $G$ are
  rational-valued. Then $S(G)$ is a reflexive and self-polar lattice simplex.
  In particular, $S(G)$ is the convex hull of the columns of the truncated table
  $V'$ of real irreducible characters.
\end{prop}

\begin{proof}
  As in \eqref{eq:ineq} we identify the affine hyperplane $x_1=1$ in $\RR^m$ with $\RR^{m-1}$ by dropping the first coordinate $x_1=1$.
  In this way the vector $e_1\in\RR^m$ corresponds to the origin in $\RR^{m-1}$.
  By Proposition~\ref{prop:simplex} the polytope $S(G)$ is a full-dimensional simplex in $\RR^{m-1}$ which contains the origin as an interior point.
  The homogeneous inward pointing normal vectors of the $m$ facets of $S(G)$ are given by the columns of the character table, which is integral.
  Each vertex, $v\in\RR^{m-1}$, of $S(G)$, is the intersection of precisely $m-1$ facets.
  That is, the homogenisation $(1,v)$ is contained in the kernel of the $m\times(m{-}1)$-matrix formed by those columns.
  Since any two columns of $X$ are orthogonal, it follows that $(1,v)$ and the remaining column, $s$, are linearly dependent.
  Let $g\in G$ be in the conjugacy class corresponding to the column $s$.
  Then $\chi_1(g)=s_1=1$, and thus $(1,v)=s$.
  It follows that the homogenised vertices of $S(G)$ agree with the homogenised inward pointing facet normal vectors.

  We conclude that $S(G)\subset\RR^{m-1}$ agrees with its own polar $S(G)^\vee$.
  Since the latter is a lattice polytope, $S(G)$ is reflexive.
\end{proof}

The next example shows that $S$-character simplices may have vertices with
non-integral coefficients.
\begin{exmp} \label{exmp:L27}
  In contrast to the situation in Example~\ref{exmp:S4} the character table
  \[
    X \ = \ 
    \begin{pmatrix}
      1 & 1 & 1 & 1 & 1 & 1 \\
      3 & -1 & 0 & 1 & \zeta+\zeta^2+\zeta^4 & \zeta^3+\zeta^5+\zeta^6 \\
      3 & -1 & 0 & 1 & \zeta^3+\zeta^5+\zeta^6 & \zeta+\zeta^2+\zeta^4 \\
      6 & 2 & 0 & 0 & -1 & -1 \\
      7 & -1 & 1 & -1 & 0 & 0 \\
      8 & 0 & -1 & 0 & 1 & 1
    \end{pmatrix}
  \]
  of the projective special linear group $\PSL_2(7)$ has non-real entries; here
  $\zeta$ is a primitive seventh root of unity, and
  $\zeta+\zeta^2+\zeta^4=(-1+\sqrt{-7})/2 \approx -0.5 + 1.32288i$.

  The second and third irreducible characters are complex conjugates; adding
  them we obtain the $5{\times}6$-matrix $U$, whose last two columns agree.
  Skipping one of them we arrive at the integral $5{\times}5$-matrix
  \[
    V \ = 
    \begin{pmatrix}
      1 & 1 & 1 & 1 & 1 \\
      6 & -2 & 0 & 2 & -1 \\
      6 & 2 & 0 & 0 & -1 \\
      7 & -1 & 1 & -1 & 0 \\
      8 & 0 & -1 & 0 & 1
    \end{pmatrix} \enspace .
  \]
  The vertices of the $S$-character simplex $S(\PSL_2(7))=\{ x\in\RR^4 \mid x (-V') \leq 1 \}$ read
  \[
    (3,6,7,8) \,,\ (-1,2,-1,0) \,,\ (0,0,1,-1) \,,\ (1,0,-1,0) \,,\ (-\tfrac{1}{2},-1,0,1)\enspace .
  \]
  In particular, $S(\PSL_2(7))$ is a not a lattice polytope.
  As in Example~\ref{exmp:S4}, the homogenised vertices correspond to the
  columns the matrix $V$, but the coefficients of the second row (which came
  from the sum of the two nonreal characters) are divided by two.
\end{exmp}

By generalising Example~\ref{exmp:L27} it is possible to obtain a vertex
description of $S(G)$ for an arbitrary finite group $G$.
From this it follows that the dilate $2\cdot S(G)$ is always a lattice polytope.
We leave the details to the reader.

\subsection{Enumerating lattice points}   \label{subsec:enum}
The task of finding relevant $S$-characters of a finite group~$G$ reduces to the following basic 3-step procedure.
First, we compute the vertices of $S(G)$ from the real character table $V$ via Cramer's rule.
Second, we enumerate all lattice points in $S(G)$.
Third, we check if the $S$-character corresponding to a given lattice point does not vanish on any element of prime power order.

Enumerating lattice points in polytopes is a topic which is both classical and
still a subject of ongoing research. For an introduction the reader is referred
to the text books by Beck--Robins \cite{BR} and Barvinok \cite{Bar}.
Two facts are decisive.
First, to detect if a polyhedron given in terms of linear inequalities contains
any lattice point, is known to be NP-complete; e.g., see \cite[Thm~18.1]{ILP}.
So there is little hope for efficient procedures.
Second, enumerating lattice points is also hard in practice; see \cite{polymake-lattice} for a computational survey, which compares several methods and their implementations.

Some of the more sophisticated methods leverage techniques from toric geometry.
A key tool is Barvinok's algorithm, which runs in polynomial time in any fixed dimension \cite[Chap.~17]{Bar}.
However, this procedure and similar ones are inherently restricted to the rational case $\KK=\QQ$.

One of the methods which works for arbitrary real number fields $\KK$ has been called \emph{project-and-lift}, and is implemented in \polymake.
We sketch the idea.
The input is a convex polytope $P$ with a double description, i.e., the list of its vertices plus the list of its facet defining inequalities (and possibly additional linear equations).
Now the project-and-lift algorithm for enumerating the lattice points in $P$ works recursively.
In the base case the polytope $P$ is one-dimensional.
Then the two vertices span a line segment; enumerating its lattice points is straightforward.
If $\dim P=d\geq 2$, then we can project out one of the coordinate directions to obtain the image $Q$ where $\dim Q=d-1$.
We can enumerate the lattice points in $Q$ by applying project-and-lift.
By tracing the (one-dimensional) fibres of the projection map we find the lattice points of $P$.
Despite the fact that lift-and-project is a very simple algorithm, it behaves rather well in practice \cite[Rule of Thumb 8]{polymake-lattice}.
Most importantly here, however, is the fact that it works for arbitrary real number fields.

\begin{rem}\label{rem:speedup}
  There are several ways to speed up the procedure sketched above.
  Observe that the value of the $S$-character $\psi$ on $g\in G$ agrees with the scalar product $\langle x, v\rangle$, where $x\in\ZZ^m$ is the vector of coefficients of $\psi$ expressed as a linear combination of irreducibles, and $v$ is the column of $V$ corresponding to the conjugacy class of $g$.
  That is to say, the $S$-character $\psi$ does not vanish on $g$ if and only if $\langle x,v\rangle>0$.
  For instance, if the character table $X$ of $G$ is rational valued, and we want to find $S$-characters not vanishing on $g$, then the condition $\langle x,v\rangle\geq 0$ in \eqref{eq:ineq} can be replaced by $\langle x,v\rangle\geq 1$ because $x$ and $v$ are integral.
\end{rem}

\section{Characters of solvable groups}   \label{sec:solv}

In this section we summarize basic aspects of character theory, with a focus on
studying $S$-characters of solvable groups. Let $G$ be a finite group,
$N\unlhd G$ a normal subgroup and set $F:= G/N$. For a class function
$\vhi$ of $G$ let $\vhi^F$ be the class function of $F$ defined by
\[
   \vhi^F(f) = \frac{|C_F(f)|}{|G|} \sum_{x \in R(f)} |x^G| \vhi(x)\qquad
   \text{for $f\in F$},
\]
where $R(f)$ is a set of representatives of $G$-classes that fuse to $f^F$
under the natural epimorphism. Then $\vhi^F$ is the projection of $\vhi$ to $F$;
here for a class function $\theta$ of $F$, let $\theta_G$ denote the inflation
of $\theta$ to $G$, that is, $\theta_G(g) = \theta(gN)$ for $g \in G$:

\begin{lem}   \label{lem:proj}
 Let $\chi\in\Irr(G)$. Then either $\chi^F\in \Irr(F)$ or $\chi^F$ is zero, and
 $$(\chi^F)_G=\begin{cases} \chi& \text{if $N\subseteq \ker(\chi)$},\\
  0& \text{otherwise}.\end{cases}$$
\end{lem}

\begin{proof}
For $\psi$ a class function on $F$ and $\vhi$ a class function on $G$ we have
\begin{eqnarray*}
  [\vhi^F, \overline{\psi}]_F & = &
  \frac{1}{|F|} \sum_{f \in F/\sim} |f^F| \vhi^F(f) \psi(f) \\
  & = & \frac{1}{|F|} \sum_{f \in F/\sim}
     |f^F| \frac{|C_F(f)|}{|G|} \sum_{x \in R(f)} |x^G| \vhi(x) \psi(f) \\
  & = & \frac{1}{|G|} \sum_{f \in F/\sim} \sum_{x \in R(f)}
     |x^G| \vhi(x) \psi_G(x) \\
  & = & \frac{1}{|G|} \sum_{y \in G/\sim} |y^G| \vhi(y) \psi_G(y) \\
  & = & [\vhi, \overline{\psi_G}]_G.
\end{eqnarray*}
Thus, letting $\psi$ run over $\Irr(F)$ we see that $\vhi^F$ is either a
character of~$F$ or zero, if $\vhi$ is a character of $G$. To conclude note that
the inflations $\psi_G$ are exactly those irreducibles of $G$ whose kernel
contains $N$.
\end{proof}

If $\psi$ is an $S$-character of $G$ then $\psi^F$ is an $S$-character of $F$,
because the trivial character of $F$ has multiplicity one in $\psi^F$, and
$\psi^F$ is real and non-negative because its values are sums of real and
non-negative numbers, by definition.

Moreover, if $\psi$ is nonzero on all elements of prime power order in $G$ then
also $\psi^F$ is nonzero on all elements of prime power order in $F$. For that,
note that for an element $g\in F$ whose order is a power of the prime $p$,
the natural epimorphism $\pi:G\to F$ sends any Sylow $p$-subgroup of
$\pi^{-1}(\langle g\rangle)$ to the Sylow $p$-subgroup of $\langle g\rangle$ and
thus the preimage of $g$ under $\pi$ contains an element of prime power order.

Recall that an $S$-character is called \emph{ordinary} if it is an actual
character.

\begin{prop}\label{prop:solvnotord}
 Each non-trivial ordinary $S$-character of a finite solvable group vanishes on
 some element of prime power order.
\end{prop}

\begin{proof}
Let $G$ be a counterexample of minimal order, and let $\psi$ be a non-trivial
ordinary $S$-character of $G$ that is nonzero on all elements of prime power
order. Since $\psi$ vanishes at some element, $G$ is not elementary abelian.

Let $N$ be a non-trivial elementary abelian (and hence proper) normal subgroup
of~$G$, and set $F = G/N$. By Lemma~\ref{lem:proj} and the subsequent remark,
$\psi^F$ is an $S$-character of~$F$ that is nonzero on all elements of prime
power order. By the minimality of~$G$, $\psi^F$ is the trivial character.

This means that $\psi$ consists of the trivial character plus irreducible
constituents $\chi$ with the property $N \not\subseteq \ker(\chi)$.
By the theorem of Clifford, the restriction of any such $\chi$ to $N$ does not
have a trivial constituent, which implies that the restriction of $\psi$ to $N$
is an $S$-character of $N$ that is nonzero on all elements of prime power order.
As above, we conclude that this restriction is the trivial character.
Since $\psi$ is an ordinary character with a trivial constituent,
$\psi$ itself is trivial.
\end{proof}

\section{Results}   \label{sec:ex}

We now discuss our computations and the examples found.
Each of these examples gives a negative answer to Serre's question \cite{Se}.
While all of them were found in a computer search with
\OSCAR~\cite{OSCAR-book,OSCAR}, in the small cases the output can easily be
verified by hand.

\begin{exmp}   \label{xplJ1A8}
We discuss the group $G = \fA_8$. It has $14$ conjugacy classes, the degrees of
its absolutely irreducible characters are
$1, 7, 14, 20, 21, 21, 21, 28, 35, 45, 45, 56, 64, 70$.
There are two pairs of complex conjugate irreducibles, which yields a search
space of about $10^{20}$ real virtual characters with the property that the
absolute values of their coefficients are bounded by the degrees.
The $S$-character simplex $S(\fA_8)$ contains exactly $3636$ lattice points,
and just one of them belongs to an $S$-character (actually an ordinary
character) $\psi$ that is nonzero on all prime power order elements. We have
$$\psi =\sum_{i=1}^{14}a_i\,\chi_i\qquad
  \text{with $(a_i)=(1, 1, 1, 1, 1, 1, 1, 1, 2, 2, 2, 3, 3, 3)$}$$
and $\psi$ takes values
$$\begin{array}{cccccccccccccccccccc}
   1a& 2a& 2b& 3a& 3b& 4a& 4b& 5a& 6a& 6b& 7a& 7b& 15a& 15b\\
\hline
  953& 9& 1& 5& 2& 1& 1& 3& 1& 0& 1& 1& 0& 0\\
\end{array}$$
on the 14 conjugacy classes of $\fA_8$.
\end{exmp}

The above example proves one half of Theorem~\ref{thm}. The other half of the
claim excludes all other finite groups with at most 14 conjugacy classes.
That proof rests on a comprehensive computation with \OSCAR, which is described
in Section~\ref{subsec:enum}. Recall that the finite groups with up to 14
conjugacy classes have been classified~\cite{VLVL85,VLVL86,VS07}.
\OSCAR is particularly well suited for such
a computation because it features tools for group and representation theory
(inherited from \GAP~\cite{GAP}), polyhedral geometry (from
\polymake~\cite{polymake-lattice}) as well as number theory. Number theory is
relevant because we employ polyhedral geometry over real number fields.

Natural candidates to check for examples are the groups whose character tables
are contained in the library of character tables, which is available in {\OSCAR}
via~\cite{GAP}.
Table~\ref{tab:ex} contains a survey of the number of $S$-characters of finite
groups not vanishing on any prime power element that we found. The columns of
this table show a name of the group, the numbers of all irreducibles, of
$\RR$-irreducibles and of $\QQ$-irreducibles, then the number of those
non-trivial $S$-characters that are faithful and nonzero on all classes of prime
power order elements, and the number of the subset of those $S$-characters that
are not ordinary characters.
For instance, there are two faithful examples for the double cover
$2.\fA_8$ of the alternating group $\fA_8$, one of them ordinary,
and the projection of each of them to $\fA_8$ is the unique example
for that group.
All known examples are in fact rational.

We already mentioned that
$\fA_8$ is the only finite group with at most $14$ conjugacy classes
for which a (not necessarily rational) example exists.  Additionally,
we verified that $\fA_8$ is also the group of smallest order with this
property whose character table occurs in the ATLAS of Finite Groups.

The multiplicities of irreducible constituents can be quite large,
the multiplicity $2236$ occurs in one of the examples for the group $M^cL$.

Several ATLAS groups do not contribute examples.
We checked all ATLAS groups with at most $40$ conjugacy classes,
a few extensions of these groups such as $2.\fA_{11}$ and $2.HS$
(which have more conjugacy classes),
and the library character tables of perfect groups with at most $40$
conjugacy classes.
The enumeration of the lattice points did not finish for the groups
$HS.2$, $O_8^-(2)$, ${}^3D_4(2)$, $G_2(4)$, $G_2(4).2$, $McL.2$, $He$,
$Ru$, and $O'N$.
For example, using \OSCAR~1.3.0-DEV with Julia~1.10.7 on an 8-core
Intel Core i7-11700K running Debian Linux, with 128G main memory,
the computation for $O'N$ ran out of memory after more than three hours.
For the group ${}^2G_2(27)$, the computation of solutions for the
\emph{rational} character table was successful but that for the
\emph{real} character table did not finish.

In the computations, it is important to force the positivity of the solutions at
the classes of prime power order elements in advance where possible (see
Remark~\ref{rem:speedup}); for
rational columns of the character table, this can be done by prescribing a value
$\geq 1$. For example, enumerating the lattice points in the $S$-simplex of the
group $\PSU_4(3)$ requires an increasing amount of space, but we quickly get an
empty set of solutions if we search a priori only for lattice points
corresponding to $S$-characters which are nonzero on elements of prime power
order.

\section{Concluding remarks}

The world of $S$-characters of finite groups seems to be a rich topic, at the
intersection of representation theory, polyhedral geometry and number theory.
While we found some interesting examples, we could not answer the questions
below.
\begin{qst}
  Are there non-rational $S$-characters that do not vanish on an element of
  prime power order?
\end{qst}

We do not have an example for a solvable group.
\begin{qst}
  Are there (necessarily not ordinary) non-trivial $S$-characters of solvable
  groups that do not vanish on any element of prime power order?
\end{qst}
For those solvable groups of order at most $1535$ that have at most $40$
$\QQ$-irreducibles, we checked that there is no rational example.
When we search for a solvable example of minimal order, Lemma~\ref{lem:proj}
yields that we have to look only for an $S$-character such that all non-trivial
constituents are faithful, in particular the centre of the group must be cyclic;
by the proof of Proposition~\ref{prop:solvnotord}, also the commutator factor
group must be cyclic.

\begin{qst}
  For a non-trivial $S$-character $\psi$ that does not vanish on any element
  of prime power order, can it happen that the projection $\psi^F$ to a
  non-trivial factor group $F$ is the trivial character of $F$?
\end{qst}


\appendix

\begin{table}
\caption{Examples}   \label{tab:ex}
\(
\begin{array}{lrrrrr}
  \toprule
  G & \# \text{classes} & \# \text{real} & \# \text{rat.} & \# \text{S-char.}
  & \# \text{virt.\ S-char.} \\
  \midrule
  \fA_{8} & 14 & 12 & 12 & 1 & 0 \\
  2.\fA_{8} & 23 & 18 & 18 & 2 & 1 \\
  \fA_{9} & 18 & 17 & 17 & 3 & 0 \\
  2.\fA_{9} & 30 & 26 & 26 & 10 & 7 \\
  \fA_{10} & 24 & 24 & 23 & 1 & 0 \\
  2.\fA_{10} & 39 & 38 & 35 & 3 & 2 \\
  \fA_{11} & 31 & 30 & 29 & 6 & 0 \\
  2.\fA_{11} & 49 & 47 & 43 & 53 & 26 \\%
  [1.6ex]
  M_{12} & 15 & 14 & 14 & 1 & 0 \\
  M_{12}.2 & 21 & 21 & 19 & 1 & 0 \\
  J_{1} & 15 & 15 & 10 & 1 & 0 \\
  J_{2} & 21 & 21 & 16 & 155 & 8 \\
  2.J_{2} & 38 & 37 & 28 & 2571 & 1280 \\
  J_{2}.2 & 27 & 27 & 26 & 12 & 1 \\
  J_{3} & 21 & 20 & 14 & 5 & 0 \\
  M^{c}L & 24 & 18 & 18 & 2588 & 65 \\
  HS & 24 & 22 & 22 & 93 & 2 \\
  2.HS & 42 & 34 & 33 & 2211 & 1094 \\
  M_{24} & 26 & 21 & 21 & 223 & 4 \\%
  [1.6ex]
  \PSL_{4}(3) & 29 & 26 & 25 & 6 & 0 \\
  \PSL_{5}(2) & 27 & 20 & 18 & 11 & 0 \\
  \PSL_{5}(2).2 & 33 & 33 & 28 & 3 & 0 \\
  \PSp_{4}(4) & 27 & 27 & 18 & 133 & 10 \\
  \PSp_{4}(4).2 & 30 & 29 & 28 & 5 & 1 \\
  \PSp_{4}(5) & 34 & 34 & 25 & 76 & 11 \\
  \PSp_{6}(2) & 30 & 30 & 30 & 1 & 1 \\
  \PSU_{4}(2) & 20 & 15 & 15 & 2 & 1 \\
  2.\PSU_{4}(2) & 34 & 24 & 24 & 1 & 1 \\
  {}^{2}F_{4}(2)' & 22 & 19 & 16 & 71 & 2 \\
  {}^{2}F_{4}(2)'.2 & 29 & 23 & 23 & 7 & 3 \\
  ^2G_2(27) & 35 & 32 & 13 & 9 & 0 \\
  G_{2}(3) & 23 & 22 & 21 & 24 & 0 \\%
  [1.6ex]
  2^{4}\!:\!\fA_{8} & 25 & 22 & 22 & 2 & 0 \\
  2^{4}.\fA_{8} & 25 & 22 & 22 & 2 & 0 \\
  2^{5}.\PSL_{5}(2) & 41 & 29 & 27 & 2696 & 118 \\
  2^{6}\!:\!\fA_{8} & 41 & 39 & 39 & 8 & 1 \\
  \bottomrule
\end{array}
\)
\end{table}


\begin{thebibliography}{131}

\bibitem{polymake-lattice}
{\sc B. Assarf, E. Gawrilow, K. Herr, M. Joswig, B. Lorenz, A. Paffenholz,
  T. Rehn}, Computing convex hulls and counting integer points with \polymake.
  \emph{Math. Program. Comput. \bf 9} (2017), 1--38.

\bibitem{Bar}
{\sc A. Barvinok}, \emph{Integer Points in Polyhedra}. European Mathematical
  Society (EMS), Z\"urich, 2008.

\bibitem{Bat}
{\sc V.V. Batyrev}, Dual polyhedra and mirror symmetry for Calabi--Yau
  hypersurfaces in toric varieties. \emph{J. Algebraic Geom. \bf3} (1994),
  493--535.

\bibitem{BR}
{\sc M. Beck, S. Robins}, \emph{Computing the Continuous Discretely}. Springer,
  New York, 2015.

\bibitem{Atl}
{\sc J.H.~Conway, R.T.~Curtis, S.P.~Norton, R.A.~Parker, R.A.~Wilson},
  \emph{Atlas of Finite Groups}. Clarendon Press, Oxford, 1985.

\bibitem{OSCAR-book}
{\sc W. Decker, C. Eder, C. Fieker, M. Horn, and M. Joswig}, \emph{The
  {C}omputer {A}lgebra {S}ystem {\OSCAR}: {A}lgorithms and {E}xamples}, volume~32
  of \emph{Algorithms and {C}omputation in {M}athematics}, Springer, 2025.

\bibitem{FKS}
{\sc B. Fein, W. M. Kantor, M. Schacher}, Relative Brauer groups. II.
  \emph{J. reine angew. Math. \bf328} (1981), 39--57.

\bibitem{Giu}
{\sc M. Giudici}, Quasiprimitive groups with no fixed point free elements of
  prime order. \emph{J. London Math. Soc. (2) \bf67} (2003), 73--84. 

\bibitem{KS}
{\sc M. Kreuzer, H. Skarke}, Complete classification of reflexive polyhedra in
  four dimensions. \emph{Adv. Theor. Math. Phys. \bf4} (2002), 1209--1230.

\bibitem{MNO}
{\sc G. Malle, G. Navarro, J. B. Olsson}, Zeros of characters of finite groups.
  \emph{J. Group Theory \bf3} (2000), 353--368.

\bibitem{polydb}
{\sc A. Paffenholz}, \polyDB: a database for polytopes and related objects.
  \emph{Algorithmic and Experimental Methods in Algebra, Geometry, and Number
  Theory}, 533--547, Springer, Cham, 2017.

\bibitem{ILP}
{\sc A. Schrijver}, \emph{Theory of Linear and Integer Programming}. Wiley,
  Chichester, 1986.
  
\bibitem{Se}
{\sc J-P. Serre}, \emph{Z\'eros de caract\`eres}. ArXiv 2312.17551v2, 2024.

\bibitem{VLVL85}
{\sc A. Vera L\'opez, J. Vera L\'opez}, Classification of finite groups
  according to the number of conjugacy classes. \emph{Israel J. Math. \bf51}
  (1985), 305--338.

\bibitem{VLVL86}
{\sc A. Vera L\'opez, J. Vera L\'opez}, Classification of finite groups
  according to the number of conjugacy classes II. \emph{Israel J. Math. \bf56}
  (1986), 188--221.

\bibitem{VS07}
{\sc A. Vera L\'opez, J. Sangroniz}, The finite groups with thirteen and
  fourteen conjugacy classes. \emph{Math. Nachr. \bf280} (2007), 676--694.
    
\bibitem{Zh}
{\sc \'E. M. Zhmud'}, On a variety of nonnegative generalized characters of a
  finite group. \emph{Ukra\"in. Mat. Zh. \bf47} (1995), 1338--1349; translation
  in \emph{Ukrainian Math. J. \bf47} (1995), 1526--1540. 

\bibitem{GAP}
{\sc The GAP~Group}, {\em GAP -- Groups, Algorithms, and Programming,
  Version 4.14.0}; 2024, \url{http://www.gap-system.org}.

\bibitem{OSCAR}
{\sc The~{OSCAR} Team}, \emph{{OSCAR} -- {O}pen {S}ource {C}omputer {A}lgebra
  {R}esearch system, version 1.2.2}; 2024,
  \url{https://www.oscar-system.org}.

\end{thebibliography}
\end{document}